\documentclass[12pt]{amsart}

\usepackage{pgf}
\usepackage{graphicx}
\usepackage{amsmath, amsthm, amssymb}

\newcommand\Pn{\mathcal{P}_n}

\newcommand\Qn{\mathcal{Q}_n}
\newcommand\Og{\mathcal{O}_{\mu;\nu}}
\newcommand\Ok{\mathbb{O}_{\mu;\nu}}

\newcommand\N{{\mathbb N}}

\newcommand\F{{\mathbb F}}

\DeclareMathOperator{\End}{End}

\DeclareMathOperator{\GL}{GL}

\DeclareMathOperator{\Sp}{Sp}

\DeclareMathOperator{\im}{im}

\newtheorem{thm}{Theorem}[section]
\newtheorem{prop}[thm]{Proposition}
\newtheorem{lem}[thm]{Lemma}
\newtheorem{cor}[thm]{Corollary}

\theoremstyle{remark}
\newtheorem{rem}[thm]{Remark}

\theoremstyle{definition}
\newtheorem{defn}[thm]{Definition}
\newtheorem{conj}{Notation}
\newtheorem{eg}[thm]{Example}
\address{\tiny Department of Mathematics, University of Oregon, Eugene, OR 97403-1222, USA}
\email{msun@uoregon.edu}
\author{Michael Sun}
\title{Point stabilisers for the enhanced and exotic nilpotent cones}
\begin{document}
\maketitle
\begin{abstract}
We give a semi-direct product decomposition of the point stabilisers for the enhanced and exotic nilpotent cones. In particular, we arrive at formulas for the number of points in each orbit over a finite field, which is in accordance with a recent conjecture of Achar and Henderson.
\end{abstract}
\section*{Introduction}

\indent In the theory of algebraic groups, there is much insight to be gained from studying the nilpotent cone $\mathcal{N}$ of an algebraic group $G$, which consists of the nilpotent elements in the Lie algebra of $G$. 

One of the main reasons for this is because the action of $G$ on $\mathcal{N}$ by conjugation produces an injective map (when $G$ is reductive) from the set of $G$-orbits $G\setminus\mathcal{N}$ to the irreducible representations of the Weyl group of $G$. This is a consequence of the Springer correspondence originally discovered by Springer \cite{springer1} in 1976 and explicitly described in all cases by Lusztig and Shoji by the early 1980s (see for example Shoji \cite{shoji}).

A well-known example of the Springer correspondence can be seen for the group $G=\GL(V)$ when $V$ is an $n$-dimensional vector space over an algebraically closed field $k$. This is a reductive group whose Lie algebra is $\End(V)$ and $\mathcal{N}$ its nilpotent elements. The $G$-orbits in $\mathcal{N}$ are in bijection with $\Pn$, the partitions of $n$, while the Weyl group of $G$ is just the Symmetric group $S_n$, whose irreducible representations are also parametrised by $\Pn$. In this case the correspondence is actually onto, which cannot be said for groups such as $K=\Sp(W)$, where $W$ is a $2n$-dimensional symplectic space over $k$.

To address this deficiency Kato \cite{kato}, in 2006, introduced what he called the \emph{exotic} nilpotent cone $W\times\mathfrak{N}_0$ where $\mathfrak{N}_0$ is the variety
$$\mathfrak{N}_0=\{y\in\End(W)\;|\;y\text{ is nilpotent and }\langle yv,v\rangle=0\text{ for all $v\in W$}\}$$
and $\langle\,,\,\rangle$ is the symplectic form on $W$.
He showed that the orbits in his cone obey some kind of \emph{exotic} Springer correspondence: 
$$K\setminus (W\times\mathfrak{N}_0)\longleftrightarrow\Qn$$
where $\mathcal{Q}_n$ denotes the bipartitions of $n$. 

A recent innovation in this direction is the introduction of the $\GL(V)$-set $V\times\mathcal{N}$, which has appeared in the work of Achar-Henderson \cite{AH}, Travkin \cite{trav} and others, and was given the name \emph{enhanced} nilpotent cone by Achar-Henderson in \cite{AH}. The enhanced cone mimics the cominatorics of the exotic cone while being a more accessible object.

For example, it was shown by Achar-Henderson \cite{AH} that there is a bijection $G\setminus (V\times\mathcal{N})\longleftrightarrow \Qn$ and that the closure ordering on these orbits agree with those on $K\setminus (W\times\mathfrak{N}_0)$ and are both given by a natural partial order on $\Qn$. These similarities are extended by a conjecture of Achar-Henderson \cite[Section 6]{AH}, which claims that the local intersection cohomology for the exotic nilpotent cone is the same as that of the enhanced nilpotent cone but with twice the degree. 

It is the immediate goal of this article to verify a prediction of the above conjecture (Corollary \ref{fini}). The main results are the explicit decompositions of the stabiliser groups for each orbit (Theorem \ref{big1} and Theorem \ref{big2}), from which the formulas for the number of $\F_q$ points in each orbit (Corollary \ref{fqpoints} and Corollary \ref{spfq}) follow. The treatment of the exotic cone (Section 3) is analogous to that of the enhanced cone (Section 2). Section 1 reviews the ordinary nilpotent cone of $\GL(V)$ and introduces notation.


\textbf{Acknowledgements}. This was the essence of my M.Sc thesis completed in September 2009 at the University of Sydney. I am greatly indebted to Anthony Henderson for his crucial comments.  
	
\section{The nilpotent cone}

A \emph{partition} is a non-increasing sequence of natural numbers $\lambda=(\lambda_1,\lambda_2,\dots)$ such that there exists a minimal $l(\lambda)\in\N$ with $\lambda_i=0$ for all $i>\l(\lambda)$. This $l$ is called the \emph{length} of $\lambda$. The \emph{weight} of $\lambda$ is defined to be $|\lambda|=\lambda_1+\lambda_2+\cdots+\lambda_l.$ If $|\lambda|=n$, we say that $\lambda$ is a partition of $n$. Denote by $\mathcal{P}$ the set of all partitions and by $\Pn$ the set of all partitions of $n$. Another standard notation for $\lambda\in\mathcal{P}$ is $\lambda=(i^{n_i})_{i\geq1}$ with $n_i=\#\{k\,|\,\lambda_k=i\}$. Let $n(\lambda)=\sum_{i\geq1}(i-1)\lambda_i$.
\begin{conj}For $\lambda\in\Pn$, write
$$\lambda=(l_h^{n_{l_h}})_{h\geq1}$$
where we impose $n_{l_h}>0$ and $l_h>l_{h+1}$. See Example \ref{note}. With this notation we define an accompanying indexing set $I_h=\{i\;|\;\lambda_i=l_h\}$ for each $h\geq1$.
\end{conj}
Now let $G=\GL(V)$, where $V$ is an $n$-dimensional vector space over an algebraically closed field $k$. Its (ordinary) \emph{nilpotent cone} is
$$\mathcal{N}=\{x\in\End(V)\;|\;x^n=0\}.$$ 
The $G$-orbits in $\mathcal{N}$ are parametrised by $\Pn$ since $G$ acts by conjugation.
\begin{defn}Let $x\in\mathcal{N}$. Call $\lambda\in\Pn$ of the reordered Jordan block sizes of $x$ the \emph{Jordan type} of $x$ and write $\mathcal{O}_{\lambda}$ for the orbit $Gx$.
\end{defn}
\begin{defn}Let $x\in\mathcal{N}$ have Jordan type $\lambda\in\Pn$. Call $\{v_{ij}|1\leq i\leq l(\lambda), 1\leq j\leq\lambda_i\}$	a \emph{Jordan basis} for $x$ if 
$$xv_{ij}=\begin{cases}0&\text{if $j=1$}\\
										v_{i,j-1}&\text{if $j>1$.}\end{cases}$$
\end{defn}
For $x\in\mathcal{N}$, fix a Jordan basis $\{v_{ij}\}$ and let $\End(V)^x=\{y\in\End(V)\;|\; yx=xy\}$, as well as $G^x=G\cap\End(V)^x$. 

\begin{conj}Let $y\in\End(V)$. Define $c_{ij}^{rs}(y)\in k$ by
$$yv_{ij}=\sum_{r,s}c_{ij}^{rs}(y)v_{rs}$$
\end{conj}
\begin{defn} For each $h\geq1$, let $V_h$ be the vector space given by the quotient
$$V_h=\frac{\ker x^{l_h}}{\ker x^{l_h-1}+\im x\cap\ker x^{l_h}}.$$
\end{defn}
Note that the image of $\{v_{i\lambda_i}\}_{i\in I_h}$ forms a basis for $V_h$ (so $\dim V_h=n_{l_h}$) and denote these cosets by $\{v_{i\lambda_i}'\}_{i\in I_h}$. $G^x$ acts on each $V_h$ to give 
$$\Psi_0: G^x\to\prod_{h\geq1}\GL(V_h)$$
which in explicit matrix terms is
$$g\mapsto \prod_{h\geq1}(c_{i\lambda_i}^{r\lambda_r}(g))_{i,r\in I_h}.$$
\begin{lem}\label{c}For $y\in\End(V)$, $y$ belongs to $\End(V)^x$ if and only if the following conditions hold
	\begin{itemize}
		\item[(i)] $c_{ij}^{rs}(y)=c_{i,j-m}^{r,s-m}(y)\text{ for }m<s\leq\lambda_r  \text{ and } m<j\leq\lambda_i$
		\item[(ii)] $c_{ij}^{rs}(y)=0\text{ if }s>j$
		\item[(iii)] $c_{ij}^{r\lambda_r}(y)=0\text{ if }j\neq\lambda_i$
	\end{itemize}
\end{lem}
\begin{proof}This is clear. 
\end{proof}
\begin{lem}\label{invert}Let $y\in\End(V)^x$. Then
$$\det(y)=\prod_{h\geq1}\det((c_{i\lambda_i}^{r\lambda_r}(y))_{i,r\in I_h})^{l_h}.$$
In particular, $y\in G^x$ if and only if $(c_{i\lambda_i}^{r\lambda_r}(y))_{i,r\in I_h}\in\GL_{n_{l_h}}(k)$ for all $h\geq1$.
\end{lem}
\begin{proof}This is clear.
\end{proof}

\begin{prop}\label{centraliser}Let $x\in\mathcal{N}$ with Jordan type $\lambda$ and let $G^x$ be the centraliser of $x$. Then
$$G^x\cong U\rtimes\prod_{h\geq1}\GL(V_h)$$
where $U$ is unipotent and isomorphic to affine space of dimension
$$|\lambda|+2n(\lambda)-\sum_{h\geq1} n_{l_h}^2.$$
\end{prop}
\begin{proof}This follows easily since $\Psi_0$ splits.
\end{proof}

\section{Orbits in the enhanced nilpotent cone}
Continuing with the notation of Section $1$, define the \emph{enhanced nilpotent cone} of $G$ to be the $G$-set $V\times\mathcal{N}$, with $G$ action 
$$g(v,x)=(gv,gxg^{-1})\quad\text{for all $g\in G$ and $(v,x)\in V\times\mathcal{N}$}.$$
Write $G^{(v,x)}$ for the stabiliser of $(v,x)$. We give a parametrisation of these orbits following Achar-Henderson\cite{AH}.
\begin{defn}A \emph{bipartition} of $n$ is a pair of partitions $(\mu;\nu)$ such that $\mu+\nu\in\Pn$. Denote by $\Qn$ the set of all bipartitions of $n$. 
\end{defn}
\begin{rem}\label{rem1} Note that if we have $(\mu;\nu)\in\Qn$ with $\lambda=\mu+\nu$, then $\lambda_i=\lambda_{i+1}$ implies $\mu_i=\mu_{i+1}$ and $\nu_i=\nu_{i+1}$ since $\mu,\nu\in\mathcal{P}$.
\end{rem}
\begin{defn}For $(\mu;\nu)\in\Qn$ and $\lambda=\mu+\nu$, define
$$b(\mu;\nu)=|\nu|+2n(\mu)+2n(\nu)=|\lambda|+2n(\lambda)-|\mu|.$$
\end{defn}
\begin{prop}\label{basis} Let $(v,x)\in V\times\mathcal{N}$, with $\lambda$ the Jordan type of $x$. There exists a Jordan basis $\{v_{ij}\}$ of $V$ such that 
$$v=\sum_{i\geq1}v_{i\mu_i}$$
where $(\mu;\lambda-\mu)\in\Qn$. Moreover this gives a one to one correspondence 
$$G\setminus (V\times\mathcal{N})\longleftrightarrow \Qn$$
where $G(v,x)\leftrightarrow(\mu;\lambda-\mu)$. Denote this orbit $\mathcal{O}_{\mu;\nu}$ where $\nu=\lambda-\mu$.
\end{prop}
\begin{proof}See \cite[Proposition 2.3]{AH}.
\end{proof}



Now fix $(\mu;\nu)\in\Qn$ with $\lambda=\mu+\nu$ and let $(v,x)\in\Og$. Also fix a Jordan basis $\{v_{ij}\}$ so that $v=\sum_{i\geq1}v_{i\mu_i}$. We may change our $v$ to be any vector in the orbit $G^xv$ and still keep $x$ in Jordan form and $G^{(v,x)}$ isomorphic. Since $\prod_{h\geq1}\GL(V_h)$ is a subgroup of $G^x$, we may rewrite $v$ as
$$v=\sum_{i(h)\leq l(\mu)}v_{i(h)\mu_{i(h)}}$$
where $i(h)$ is the smallest $i$ in $I_h$ and define $v_h\in V_h$ by
$$v_h=v_{i(h)\lambda_{i(h)}}'.$$ 
\begin{conj}Let $(\mu;\nu)\in\Qn$ and $\mu+\nu=\lambda=(l_h^{n_{l_h}})_{h\geq1}$. Write
$$\begin{array}{cc}\mu=(j_h^{n_{l_h}})_{h\geq1},&\nu=(k_h^{n_{l_h}})_{h\geq1}\end{array}$$
with $j_h\geq j_{h+1}$ and $k_h\geq k_{h+1}$, which makes sense by Remark \ref{rem1}. Set $k_0=\infty$. See Example \ref{note}. Also define the indexing set
$$J=\{h\geq1\;|\;j_{h}>j_{h+1}\text{ and }k_h<k_{h-1}\}.$$
\end{conj}
\begin{lem}\label{stab}If $h\in J$, then $G^{(v,x)}v_h=v_h$.
\end{lem}
\begin{proof}
This is straightforward.
\end{proof}
Now let $h\in J$ and $\partial_h:\GL(V_h)^{v_h}\to\GL_{n_{l_h}-1}(k)$ the map induced by the action of $\GL(V_h)^{v_h}$ on $V_h/kv_h$ with basis $\{v_{j\lambda_j}'+kv_h\}_{j\in I_h\setminus\{i(h)\}}$. 
\begin{defn}By Lemma \ref{stab} we can compose $\Psi_0$ from Section 1 with $\prod_{h\in J}\partial_h$:
$$\Psi: G^{(v,x)}\to\left(\prod_{h\in J}\GL_{n_{l_h}-1}(k)\right)\times\left(\prod_{h\notin J}\GL_{n_{l_h}}(k)\right)$$
In explicit matrix terms, the map is 
$$g\mapsto\prod_{h\in J}(c_{i\lambda_i}^{r\lambda_r}(g))_{i,r\in I_h\setminus\{i(h)\}}\prod_{h\notin J}(c_{i\lambda_i}^{r\lambda_r}(g))_{i,r\in I_h}.$$
\end{defn}
	
	
	
			
\begin{lem}\label{ker}$\ker\Psi$ is isomorphic to affine space of dimension
$$b(\mu;\nu)-\sum_{h\notin J}n_{l_h}^2-\sum_{h\in J}(n_{l_h}-1)^2.$$
\end{lem}
\begin{proof}$\ker\Psi$ is defined by:\\
\begin{itemize}
\item$|\lambda|+2n(\lambda)-\sum_{i\geq1} n_i^2$ free variables from $\ker\Psi_0$ (Prop. \ref{centraliser}).\\
\item$\sum_{h\in J}(n_{l_h}-1)$ free variables from $\ker\prod_{h\in J}\partial_h$.\\
\item$\sum_{h\in J}n_{l_h}(j_h-1)+\sum_{h\notin J}n_{l_h}j_h$ linear relations from fixing $v$.\end{itemize}
These independently combine to give the result. See Example \ref{note}.
\end{proof}
We now provide a subgroup of $G^{(v,x)}$ isomorphic to $\im\Psi$. 
\begin{conj}For $t>0$ define
\begin{eqnarray*}R_t&=&\{h\;|\;j_h=j_{h+1}=\dots=j_{h+t}\neq j_{h+t+1}\}\\
L_t&=&\{h\;|\;k_h=k_{h-1}=\dots=k_{h-t}\neq k_{h-t-1}\}.\end{eqnarray*}
\end{conj}
\begin{defn}\label{H}Let $H$ be the subgroup of $G^{(v,x)}$ defined by the following relations. For each $r\in I_h$, we set $c_{i\lambda_i}^{rj}=0$ if $j\neq\lambda_r$ and the following
\begin{itemize}\item[]
\item\underline{$h\in R_t$, $j_h\neq0$.}
$$c_{i\lambda_i}^{r\lambda_r}=\begin{cases}\delta_{ri(h)}-c_{i(h)\lambda_{r}}^{r\lambda_r}&\text{if $i=i(h+t)$.}\\
												0&\text{otherwise, unless $i\in I_h$.}\end{cases}$$
\item\underline{$j_{h}\neq j_{h+1}$, $h\in L_t$.}
$$c_{i\lambda_i}^{r\lambda_r}=\begin{cases}\delta_{ri(h)}-c_{i(h)\lambda_r}^{r\lambda_r}&\text{if $i=i(h-t)$.}\\
												0&\text{otherwise, unless $i\in I_h$.}\end{cases}$$		
\item\underline{$h\in J$.}
$$c_{i\lambda_i}^{r\lambda_r}=0\text{ unless $i,r\in I_h\setminus\{i(h)\}$ or $i=r=i(h)$.}$$
\item\underline{$j_h=0$.}
$$c_{i\lambda_i}^{r\lambda_r}=0\text{ unless $i\in I_h$.}$$\end{itemize}
\end{defn}
\begin{prop}\label{sub}The restriction of $\Psi$ to $H$ gives 
$$H\cong\left(\prod_{h\in J}\GL_{n_{l_h}-1}(k)\right)\times\left(\prod_{h\notin J}\GL_{n_{l_h}}(k)\right).$$
\end{prop}
\begin{proof}The idea is to define elements of $H$ with as many zero entries as possible without changing the image of $\Psi$, while keeping just enough entries non-zero to accomodate the condition for stabilising $v$. For a clear way to see all the properties claimed, reorder the basis to be $v_{11},\dots,v_{l(\lambda)1},v_{12},v_{22},v_{32}$ etc. so that the elements in $H$ are essentially block diagonal. See Example \ref{note}. 
\end{proof}

\begin{thm}\label{big1}Suppose $(v,x)\in\Og$, then
$$G^{(v,x)}\cong U\rtimes\left(\left(\prod_{h\in J}\GL_{n_{l_h}-1}(k)\right)\times\left(\prod_{h\notin J}\GL_{n_{l_h}}(k)\right)\right)$$
where $U$ is unipotent and isomorphic to affine space of dimension
$$b(\mu;\nu)-\sum_{h\notin J} n_{l_h}^2-\sum_{h\in J}(n_{l_h}-1)^2.$$
\end{thm}
\begin{proof}
This follows from Lemma \ref{ker} and Proposition \ref{sub}.
\end{proof}
We will consider the following example to summarise our results. 
\begin{eg}\label{note}
Let $\lambda=(2,2,1,1)=(2^2,1^2)$ and $\mu=(1,1,1,1)=(1^2,1^2)$. Elements in $G^{(v,x)}$ can be written
$$\left(\begin{array}{cc|cc|c|c}
		a&t_1&b&t_5&t_9&t_{11}\\
		0&a&0&b&0&0\\\hline
		c&t_2&d&t_6&t_{10}&t_{12}\\
		0&c&0&d&0&0\\\hline
		0&t_3&0&t_7&x&y\\\hline
		0&t_4&0&t_8&z&w
		\end{array}\right), \begin{array}{rc}
a+t_9&=1\\c+t_{10}&=0\\x&=1\\z&=0\end{array}$$
where the relations are from stabilising $v=(1,0,0,0,1,0)^t$. 
Elements of $\ker\Psi$ have the form
$$\left(\begin{array}{cc|cc|c|c}
		1&t_1&0&t_5&t_9&t_{11}\\
		0&1&0&0&0&0\\\hline
		0&t_2&1&t_6&t_{10}&t_{12}\\
		0&0&0&1&0&0\\\hline
		0&t_3&0&t_7&1&y\\\hline
		0&t_4&0&t_8&0&1
		\end{array}\right), \begin{array}{c}t_9=0\\t_{10}=0.\end{array}$$
So $\ker\Psi\cong\mathbb{A}^{11}$ as varieties. Elements of $H$ have the form
$$\left(\begin{array}{cc|cc|c|c}
		a&0&b&0&1-a&0\\
		0&a&0&b&0&0\\\hline
		c&0&d&0&-c&0\\
		0&c&0&d&0&0\\\hline
		0&0&0&0&1&0\\\hline
		0&0&0&0&0&w
		\end{array}\right)\stackrel{\textbf{reorder basis}}{\longrightarrow}\left(\begin{array}{cc|c|c|cc}
		a&b&1-a&0&0&0\\
		c&d&-c&0&0&0\\\hline
		0&0&1&0&0&0\\\hline
		0&0&0&w&0&0\\\hline
		0&0&0&0&a&b\\
		0&0&0&0&c&d
		\end{array}\right).$$		
$$$$			
So $H\cong\GL_2\times\GL_{2-1}$ and $G^{(v,x)}\cong U\rtimes(\GL_2\times\GL_{2-1})$. where $U\cong\mathbb{A}^{11}$.
\end{eg}
Now suppose $\F_q$ is a finite subfield of $k$, and $(v,x)$ is an $\F_q$-rational point. The stabiliser $G^{(v,x)}$ is defined over $\F_q$ and we have:
\begin{cor}\label{size}$$|G^{(v,x)}(\F_q)|=q^{b(\mu;\nu)}\prod_{h\in J}\varphi_{n_{l_h}-1}(q^{-1})\prod_{h\notin J}\varphi_{n_{l_h}}(q^{-1})$$
where $\varphi_m(t)=\prod_{r=1}^{m}(1-t^r)$. 
\end{cor}
\begin{proof}As $|\GL_n(\F_q)|=q^{n^2}\varphi_n(q^{-1})$, the result is immediate.
\end{proof}

\begin{cor}\label{fqpoints} 
$$|\mathcal{O}_{\mu;\nu}(\F_q)|=\frac{q^{n^2-b(\mu;\nu)}\varphi_n(q^{-1})}{\prod_{h\in J}\varphi_{n_{l_h}-1}(q^{-1})\prod_{h\notin J}\varphi_{n_{l_h}}(q^{-1})}.$$
\end{cor}
\begin{proof}The stabiliser $G^{(v,x)}$ is connected (a consequence of Theorem \ref{big1}) and so $\Og(\F_q)$ is a single orbit of $G(\F_q)$. 
\end{proof}
	
\section{Orbits in the exotic nilpotent cone}
Let $W$ be a $2n$-dimensional vector space with a symplectic form $\langle\;,\;\rangle$ over $k$ and let $\mathfrak{N}_0$ be the variety of self adjoint nilpotent endomorphisms of $W$ as defined in the introduction.
Define the \emph{exotic nilpotent cone} of $K=\Sp(W,\langle\;,\;\rangle)$ to be the variety $W\times\mathfrak{N}_0$, which has a $K$ action given by
$$g(w,y)=(gw,gyg^{-1})\quad\text{for all $g\in K$ and $(w,y)\in W\times\mathfrak{N}_0$}.$$
Denote by $K^{(w,y)}$ the stabiliser of $(w,y)$. We adapt the parametrisation of Lemma \ref{basis} to $K\setminus W\times\mathfrak{N}_0$ following Achar-Henderson \cite{AH}.

Identify $V$ with a subspace of $W$ by writing $W=V\oplus V^*$, where $V^*$ is the dual space of $V$. The symplectic form $\langle\;,\;\rangle$ is given by
$$\langle(v,f),(v',f')\rangle=f'(v)-f(v')\quad\text{ for all $(v,f),(v',f')\in V\oplus V^*$.}$$
From here we can identify $\mathcal{N}$ with $\{(x,x^t)\,|\,x\in\mathcal{N}\}\subset\mathfrak{N}_0$ and $\GL(V)$ with $\{g\in K\,|\,gV=V$ and $gV^*=V^*\}\subset K$ to give
$$V\times\mathcal{N}\subset W\times\mathfrak{N}_0\subset W\times\mathcal{N}(W)$$
where $\mathcal{N}(W)$ are the nilpotent endomorphisms of $W$.
\begin{thm}\label{oo}There is a one to one correspondence 
$$K\setminus W\times\mathfrak{N}_0\leftrightarrow\Qn$$ and write $\mathbb{O}_{\mu;\nu}$ for the orbit corresponding to $(\mu;\nu)\in\Qn$. Furthermore, under the identification given above, this parametrisation satisfies
$$\Og\subset\Ok\subset\mathcal{O}_{\mu\cup\mu;\nu\cup\nu}.$$
\end{thm}
\begin{proof}Originally proved by Kato (\cite{kato} and for characteristic $2$ in \cite[Corollary 4.3]{kato2}) and was restated as above by Achar-Henderson \cite[Theorem 6.1]{AH}.
\end{proof}

Let $(\mu;\nu)\in\Qn$ with $\lambda=\mu+\nu$, $(v,x)\in\mathcal{O}_{\mu;\nu}$, $\{v_{ij}\}$ a basis  as in Lemma \ref{basis} and write $(v,x)$ in the same form as in Section 2.3. Now by Theorem \ref{oo} we can write $$(w,y)=((v,0),(x,x^t))\in\mathbb{O}_{\mu;\nu}.$$
Letting $\{v_{ij}^*\}$ denote the dual basis of $\{v_{ij}\}$ for $V^*$, we define a basis $\{w_{ij}\}$ for $W=V\oplus V^*$ by
$$w_{ij}=\begin{cases}v_{ij}&\text{if $i\leq l(\lambda)$}\\
											v_{\bar{i},\lambda_{i}-j+1}^*&\text{otherwise}\end{cases}$$
where we define $\bar{i}=2l(\lambda)-i+1$ for $i\geq0$, and $\lambda_i=\lambda_{\bar{i}}$, $\mu_i=\mu_{\bar{i}}$ for $i>l(\lambda)$. 
\begin{conj}Redefine $I_h$ to be $\{i\;|\;\lambda_i=l_h\}$ in this expanded definition of $\lambda_i$. $l(\lambda)$ and $l(\mu)$ will still refer to the lengths of the original partitions $\lambda$ and $\mu$. Define $b_{ij}^{rs}(z)\in k$ by 
$$zw_{ij}=\sum_{r,s}b_{ij}^{rs}(z)w_{rs}\quad\text{for all $z\in\End(W)$.}$$
For each $h\geq1$ let $j(h)$ be the largest index in $I_h$. That is, $j(h)=\overline{i(h)}$.
\end{conj}We may now write
$$w=\sum_{i(h)\leq l(\mu)}w_{i(h)\mu_{i(h)}}.$$
With respect to $\{w_{ij}\}$, $y$ is in Jordan form with block sizes symmetric in the off-diagonal. Lemmas \ref{c} and \ref{invert} still apply.
\begin{lem}\label{spinvert}Let $z\in K^y$. Then $(b_{i\lambda_i}^{r\lambda_r}(z))_{i,r\in I_h}\in\Sp_{2n_{l_h}}(k)$ for all $h\geq1$
\end{lem}
\begin{proof}This is straightforward.
\end{proof}
Define $W_h$ similarly to $V_h$ by
$$W_h=\frac{\ker y^{l_h}}{\ker y^{l_h-1}+\im y\cap\ker y^{l_h}}$$
and note that the cosets $\{w_{i\lambda_i}'\}_{i\in I_h}$ with representatives $\{w_{i\lambda_i}\}_{i\in I_h}$ form a basis for $W_h$.
Let $w_h$ be defined by
$$w_h=w_{i(h)\lambda_{i(h)}}'.$$
For each $h>0$, $K^{(w,y)}$ acts on $W_h$ to give a map $\widetilde{\Psi}_0: K^{(w,y)}\to\GL(W_h)$, which preserves the symplectic form on $W_h$ by Lemma \ref{spinvert}.
\begin{lem}\label{spstab}If $h\in J$ then $K^{(w,y)}w_h=w_h$.
\end{lem}
\begin{proof}This is identical to the proof of Lemma \ref{stab}. 
\end{proof}
Now let $h\in J$ and $\widetilde{\partial}_h:\Sp(W_h)^{w_h}\to\Sp_{2n_{l_h}-2}(k)$ the map induced by the action of $\Sp(W_h)^{w_h}$ on $w_h^{\bot}/kw_h$ with basis $\{w_{j\lambda_j}'+kw_h\}_{j\in I_h\setminus\{i(h),j(h)\}}$. 
\begin{defn}Compose $\widetilde{\Psi}_0$ with $\prod_{h\in J}\widetilde{\partial}_h$ to give
$$\widetilde{\Psi}:K^{(w,y)}\to\left(\prod_{h\in J}\Sp_{2n_{l_h}-2}(k)\right)\times\left(\prod_{h\notin J}\Sp_{2n_{l_h}}(k)\right).$$
In explicit matrix terms, 
$$g\mapsto\prod_{h\in J}(b_{i\lambda_i}^{r\lambda_r}(g))_{i,r\in I_h\setminus\{i(h),j(h)\}}\prod_{h\notin J}(b_{i\lambda_i}^{r\lambda_r}(g))_{i,r\in I_h}.$$
\end{defn}
\begin{lem}\label{spker}$\ker\widetilde{\Psi}$ is isomorphic to affine space of dimension
$$n+2b(\mu;\nu)-\sum_{h\in J}[2(n_{l_h}-1)^2+(n_{l_h}-1)]-\sum_{h\notin J}(2n_{l_h}^2+n_{l_h})$$
\end{lem}
\begin{proof}The key difference between this and Lemma \ref{ker} is treating the symplectic condition. We note that the presence of the 1s on the diagonal will mean that imposing the condition that the columns form a symplectic basis will just simply reduce the number of free variables by substitution. See Example \ref{spsmall}.
\end{proof}
\begin{eg}\label{spsmall} Let $\lambda=(2,1)$ and $\mu=(1,1)$. Elements $z\in K^{(w,y)}$ have the form
$$z=\left(\begin{array}{cc|c|c|cc}
a&t_{1}&t_{5}&t_{7}&b&t_{9}\\
0&a&0&0&0&b\\\hline
0&t_{2}&1&\beta&0&t_{10}\\\hline
0&t_{3}&0&1&0&t_{11}\\\hline
c&t_{4}&t_{6}&t_{8}&d&t_{12}\\
0&c&0&0&0&d\end{array}\right),\widetilde{\Psi}(z)=\left(\begin{pmatrix}a&b\\c&d\end{pmatrix},\begin{pmatrix}1\end{pmatrix}\right).$$
where stabilising $w=(1,0,1,0,0,0)^t$ means $a+t_5=1$, $c+t_6=0$, while the symplectic conditions, which determine the variables $t_2$, $t_3$, $t_{10}$, $t_{11}$ and either $t_4$ or $t_{12}$ depending on whether $a$ or $b$ is non-zero.
Elements of $\ker\widetilde{\Psi}$ have the form
$$\left(\begin{array}{cc|c|c|cc}
1&t_{1}&t_5&t_{7}&0&t_{9}\\
0&1&0&0&0&0\\\hline
0&t_{2}&1&\beta&0&t_{10}\\\hline
0&t_{3}&0&1&0&t_{11}\\\hline
0&t_{4}&t_{6}&t_{8}&1&t_{12}\\
0&0&0&0&0&1\end{array}\right), \begin{array}{cl}
t_{1}&=-t_{12}+t_{3}t_{10}-t_{2}t_{11}\\
t_{2}&=-t_8\\
t_{3}&=t_6\\
t_{5}&=t_6=0\\
t_{7}&=t_{10}.
\end{array}$$
So $\ker\widetilde{\Psi}\cong\mathbb{A}^{6}$ as varieties.
\end{eg}

\begin{defn}Let $\widetilde{H}$ be the subgroup of $K^{(w,y)}$ defined by the following relations. For each $r\in I_h$, let $b_{i\lambda_i}^{rj}=0$ if $j\neq\lambda_r$ and 
\begin{itemize}\item[]
\item\underline{$h\in R_t$, $j_h\neq0$.}
$$b_{i\lambda_i}^{r\lambda_r}=\begin{cases}\delta_{ri(h)}-b_{i(h)\lambda_r}^{r\lambda_r}&\text{if $i=i(h+t)$}\\
																						0&\text{otherwise, unless $i\in I_h$.}\end{cases}$$
\item\underline{$j_{h}\neq j_{h+1}$, $h\in L_t$.}
$$b_{i\lambda_i}^{r\lambda_r}=\begin{cases}\delta_{ri(h)}-b_{i(h)\lambda_r}^{r\lambda_r}&\text{if $i=i(h-t)$}\\
													  								0&\text{otherwise, unless $i\in I_h$.}\end{cases}$$	
\item\underline{$h\in J$.}
$$b_{i\lambda_i}^{r\lambda_r}=\begin{cases}1&\text{if $r=i=i(h)$ or $j(h)$}\\
																					 s_{i\lambda_i}^{r\lambda_r}&\text{if $i\notin I_h$, $r=j(h)$}\\
																					 0&\text{otherwise, unless $i,r\in I_h\setminus\{i(h),j(h)\}$.}\end{cases}$$
\item\underline{$j_h=0$.}\quad $b_{i\lambda_i}^{r\lambda_r}=0$ unless $i\in I_h$.
\end{itemize}
where the $s_{i\lambda_i}^{r\lambda_r}$ are determined by the symplectic conditions.								
\end{defn}
\begin{prop}\label{spsub}Restricting $\widetilde{\Psi}$ to $\widetilde{H}$ gives
$$\widetilde{H}\cong\left(\prod_{h\in J}\Sp_{2n_{l_h}-2}(k)\right)\times\left(\prod_{h\notin J}\Sp_{2n_{l_h}}(k)\right).$$
\end{prop}
\begin{proof}The idea of proof is very similar to that of Proposition \ref{sub} with the key difference being to keep extra entries non-zero to also make elements in $\widetilde{H}$ symplectic. Using free variables this way is possible essentially because symplectic matrices have determinant $1$. The procedure for reordering the basis is as follows. First relabel the basis $\{w_{ij}\}$ so that it is a Jordan basis for $y$, with the same relative order of the Jordan blocks of each size as before, and then reorder this newly labelled basis as in the proof of Proposition \ref{sub}. See Example \ref{note2}. 
\end{proof}
\begin{eg}\label{note2}Let $\lambda=(2,2,1,1)$ and $\mu=(1,1,1,1)$. Then $H$ consists of elements of the form
$$\left(\begin{array}{cc|cc|c|c|c|c|cc|cc}
a_{11}&0&a_{12}&0&x_{1}&0  &0&0&a_{13}&0&a_{14}&0\\
0&a_{11}&0&a_{12}&0&0      &0&0&0&a_{13}&0&a_{14}\\ \hline
a_{21}&0&a_{22}&0&x_{2}&0&0&0&a_{23}&0&a_{24}&0\\
0&a_{21}&0&a_{22}&0&0      &0&0&0&a_{23}&0&a_{24}\\\hline
0&0&0&0&1&0&0&0&0&0&0&0\\\hline
0&0&0&0&0&\alpha_{22}&\alpha_{23}&0&0&0&0&0\\\hline
0&0&0&0&0&\alpha_{32}&\alpha_{33}&0&0&0&0&0\\\hline
0&s_{1}&0&s_2&0&0&0&1&0&s_{3}&0&s_{4}\\\hline
a_{31}&0&a_{32}&0&x_{3} &0   &0&0&a_{33}&0&a_{34}&0\\
0&a_{31}&0&a_{32}&0&0      &0&0&0&a_{33}&0&a_{34}\\\hline
a_{41}&0&a_{42}&0&x_{4}&0 &0&0&a_{43}&0&a_{44}&0\\
0&a_{41}&0&a_{42}&0&0      &0&0&0&a_{43}&0&a_{44}\end{array}\right)$$
where the $x_i$'s are to stabilise $w$ and the $s_i$'s make it symplectic:
$$\begin{array}{llll}
x_{1}=1-a_{11}&x_{3}=-a_{31}&s_{1}=-a_{41}&s_{3}=-a_{43}\\
x_{2}=-a_{21}&x_{4}=-a_{41}&s_{2}=-a_{42}&s_{4}=1-a_{44}.
\end{array}$$
Change the basis as in the proof of Proposition \ref{spsub} to get
$$\left(\begin{array}{cccc|c|cc|c|cccc}
a_{11}&a_{12}&a_{13}&a_{14}&x_{1}&0  &0&0&0&0&0&0\\
a_{21}&a_{22}&a_{23}&a_{24}&x_{2}&0      &0&0&0&0&0&0\\ 
a_{31}&a_{32}&a_{33}&a_{34}&x_{3}&0&0&0&0&0&0&0\\
a_{41}&a_{42}&a_{43}&a_{44}&x_{4}&0      &0&0&0&0&0&0\\\hline
0&0&0&0&1&0&0&0&0&0&0&0\\\hline
0&0&0&0&0&\alpha_{22}&\alpha_{23}&0&0&0&0&0\\
0&0&0&0&0&\alpha_{32}&\alpha_{33}&0&0&0&0&0\\\hline
0&0&0&0&0&0&0&1&s_1&s_{2}&s_3&s_{4}\\\hline
0&0&0&0&0&0      &0&0&a_{11}&a_{12}&a_{13}&a_{14}\\
0&0&0&0&0&0      &0&0&a_{21}&a_{22}&a_{23}&a_{24}\\
0&0&0&0&0&0      &0&0&a_{31}&a_{32}&a_{33}&a_{34}\\
0&0&0&0&0&0      &0&0&a_{41}&a_{42}&a_{43}&a_{44}\end{array}\right)$$
which is isomorphic to $\Sp_{4}(k)\times\Sp_{4-2}(k).$
\end{eg}
\begin{thm}\label{big2}Suppose $(w,y)\in\Ok$, then 
$$K^{(w,y)}\cong U\rtimes\left(\left(\prod_{h\in J}\Sp_{2n_{l_h}-2}(k)\right)\times\left(\prod_{h\notin J}\Sp_{2n_{l_h}}(k)\right)\right)$$
where $U$ is unipotent and isomorphic to affine space of dimension
$$n+2b(\mu;\nu)-\sum_{h\in J}[2(n_{l_h}-1)^2+(n_{l_h}-1)]-\sum_{h\notin J}(2n_{l_h}^2+n_{l_h}).$$
\end{thm}
\begin{proof}This follows from Lemma \ref{spker} and Proposition \ref{spsub}.
\end{proof}

Now suppose $\F_q$ is a finite subfield of $k$, and $(w,y)$ is an $\F_q$-rational point. The stabiliser $K^{(w,y)}$ is defined over $\F_q$ and we have:
\begin{cor}\label{spsize}
$$|K^{(w,y)}(\F_q)|=q^{n+2b(\mu;\nu)}\prod_{h\in J}\varphi_{n_{l_h}-1}(q^{-2})\prod_{h\notin J}\varphi_{n_{l_h}}(q^{-2})$$
where $\varphi_m(t)=\prod_{r=1}^{m}(1-t^r).$ 
\end{cor}
\begin{proof}As $|\Sp_{2n}(\F_q)|=q^{n+2n^2}\varphi_n(q^{-2})$, the result is immediate.
\end{proof}
\begin{cor}\label{spfq}
$$|\mathbb{O}_{\mu;\nu}(\F_q)|=\frac{q^{2n^2-2b(\mu;\nu)}\varphi_n(q^{-2})}{\prod_{h\in J}\varphi_{n_{l_h}-1}(q^{-2})\prod_{h\notin J}\varphi_{n_{l_h}}(q^{-2})}.$$
\end{cor}
\begin{proof}Identical to Corollary \ref{fqpoints}.
\end{proof}

\begin{cor}\label{fini}$|\mathbb{O}_{\mu;\nu}(\F_q)|=|\mathcal{O}_{\mu;\nu}(\F_{q^2})|.$\qed 
\end{cor}
\ \\

 \address
\email

\begin{thebibliography}{10}
 \bibitem{AH}  
   {\sc P.~Achar, A.~Henderson}, 
   {\em Orbit closures in the enhanced nilpotent cone}, 
   Advances in Mathematics 219 (2008), 27-62.
 \bibitem{kato}  
   {\sc S.~Kato}, 
   {\em An exotic Deligne-Langlands correspondence for symplectic groups}, 
   Duke Math Journal Volume 148, Number 2(2009), 305-371.  
 \bibitem{kato2}  
   {\sc S.~Kato}, 
   {\em Deformations of nilpotent cones and Springer correspondences}. Preprint (2006). http://arxiv.org/abs/0801.3707.

 \bibitem{shoji}  
   {\sc T.~Shoji}, 
   {\em On the Springer representations of the Weyl groups of classical algebraic groups}, 
   Comm. Algebra {\bf 7} (1979), no. 16, 1713--1745.  
 \bibitem{springer1}  
   {\sc T.~Springer}, 
   {\em Trigonometric sums, Green functions of finite groups and representations of Weyl groups}, 
   Invent. Math. 36 (1976), 173-207.
 \bibitem{trav}  
   {\sc R.~Travkin}, 
   {\em Mirabolic Robinson-Schensted-Knuth correspondence}, 
   Selecta Math. (N.S.) 14 (2009), no. 3-4, 727--758.   
 
 \end{thebibliography}
 \end{document}